\documentclass[a4paper,12pt]{article}

\usepackage{amsfonts}
\usepackage{amsmath}
\usepackage{amsthm}
\usepackage{hyperref}

\newtheorem{theorem}{Theorem}
\newtheorem{lemma}[theorem]{Lemma}
\newtheorem{corollary}[theorem]{Corollary}

\DeclareMathOperator{\sympeak}{sp}
\DeclareMathOperator{\nonsympeak}{nsp}
\DeclareMathOperator{\SP}{SP}
\DeclareMathOperator{\NSP}{NSP}

\begin{document}
\title{Counting symmetric and non-symmetric peaks in a set partition}

\author{Walaa Asakly and Noor Kezil\\
Department of Mathematics, Braude college, Karmiel, Israel\\
Department of Mathematics, University of Haifa , Haifa, Israel\\
{\tt walaaa@braude.ac.il}\\
{\tt nkizil02@campus.haifa.ac.il}}
\date{\small }
\maketitle

\begin{abstract}
The aim of this paper is to derive explicit formulas
for two statistics. The first one is the total number of symmetric peaks in a set partition of $[n]$
with exactly $k$ blocks, and the second one is the total number of non-symmetric peaks in a set partition of $[n]$
with exactly $k$ blocks. We represent these results in two ways: first, by using the theory of generating functions,
and second, by using combinatorial tools.
\medskip

\noindent{\bf Keywords}: Symmetric peaks, Non-symmetric peaks, Set partitions,  Stirling numbers
of the second kind and Generating functions.

\end{abstract}

\section{Introduction}
A {\em partition} $\Pi$ of set $[n]=\{1,2,\ldots,n\}$ of size
$k$ ({\em a partition of $[n]$ with exactly $k$ blocks}) is a collection $\{B_1,B_2,\ldots,B_k\}$
that satisfies: $\emptyset\neq B_i \subseteq [n]$ for all $i$, $B_i\bigcap B_j=\emptyset$ for $i\neq j$
and $\bigcup_{i=1}^{k}B_i=[n]$.
The elements $B_i$ are called {\em blocks}. We use the assumption that blocks
are listed in an increasing order of their minimal elements, that is, $\min B_1 <
\min B_2 < \cdots< \min B_k$.
We let $P_{n,k}$ denote the set of all partitions of $[n]$ with exactly $k$ blocks see \cite{S}. Note that
  $|P_{n,k}|=S_{n,k}$, which are known as the {\em Stirling numbers of the second kind} (see A008277 in \cite{Sl}).
A partition $\Pi$ can be written as
$\pi_1\pi_2\cdots\pi_n$, where $i\in B_{\pi_i}$ for all $i$. This form is called the {\em canonical sequential form}. For
example, $\Pi=\{\{12\},\{3\}\}$ is a partition of $[3]$, the
canonical sequential form is $\pi=112$. Then we can see that any partition
$\Pi$ of set $[n]$ with exactly $k$ blocks can be represented as a word
of length $n$ over alphabet $[k]$  ({\em a word $\omega$ of length $n$ over alphabet $[k]$}
 is an element of $[k]^n$, see \cite{HM}) of the form  $1(1)^*2(2)^*\cdots k(k)^*$, where $(m)^*$ is a word
over alphabet $[m]$. These words are called {\em restricted growth functions}. For other properties of
set partitions see \cite{TM}.

The study of patterns in combinatorial structures is very popular. For example, Kitaev \cite{K} researched patterns
in permutations and words. For more examples of patterns in various combinatorial structures, refer to \cite{K1,K2,MM,Sh}.
Given $\pi=\pi_1\pi_2\cdots\pi_n \in [k]^n$, a pattern $\pi_i\pi_{i+1}\pi_{i+2}$ is called a {\em peak} at $i$ if it satisfies $\pi_i<\pi_{i+1}>\pi_{i+2}$ for $1\leq i\leq n-2$. For instance, let $\pi=1322141251$ be a word in $[5]^{10}$. It has $3$ peaks, which are $132$, $141$ and $251$. Mansour and Shattuck
\cite{MS} determined the number of peaks in all words of length $n$ over alphabet $[k]$, and derived the number of peaks over
$P_{n,k}$.
 We say that $\pi=\pi_1\pi_2\cdots\pi_n \in [k]^n$ contains  {\em a symmetric peak}, if
exists $1\leq i \leq n-2$ such that
$\pi_i\pi_{i+1}\pi_{i+2}$ is a peak, and
$\pi_i=\pi_{i+2}$, and we say that $\pi$ contains {\em a non-symmetric peak}, if exists $1\leq i \leq n-2$ such that
$\pi_i\pi_{i+1}\pi_{i+2}$ is a peak and
$\pi_i\neq \pi_{i+2}$. In the case of the example $\pi=1322141251$, there is one symmetric peak $141$, and two non-symmetric peaks, namely $132$ and $251$.
 Asakly \cite{A} determined the number of symmetric and non-symmetric peaks in all words of length $n$ over alphabet $[k]$. More recently,
 other researchers have studied symmetric and non-symmetric peaks in different combinatorial structures. For additional information refer to \cite{FR, SSZ}.
In this paper, we aim to determine the total number of symmetric and
non-symmetric peaks over $P_{n,k}$. Mansour, Shattuck and Yan \cite{MSS} found the total number of occurrences of the subword pattern $121$ (a symmetric peak) in all the partitions of $[n]$ with exactly $k$ blocks. Additionally, they found the number of occurrences of the subword patterns $231$ and $321$ (non-symmetric peaks) in all the partitions of $[n]$ with exactly $k$ blocks. We present two alternative proofs that yield the same results: the first one by using the theory of generating functions and the results that Asakly \cite{A} obtained, and the second one by using combinatorial tools.

\vspace{0.5cm}

\section{Counting symmetric peaks}
\subsection{The ordinary generating function for the number of partitions with exactly $k$ blocks according to the number of
symmetric peaks}
Let $\sympeak(\pi)$ denote the number of symmetric peaks in partition $\pi$.
Let $\SP_k(x,q)$ be the ordinary generating function for the number of partitions of $[n]$ with exactly $k$ blocks
according to the number of symmetric peaks, that is
$$\SP_k(x,q)=\sum_{n\geq k}x^n\sum_{\pi \in P_{n,k}}q^{\sympeak(\pi)}.$$

\begin{theorem}\label{Th1}
 The generating function for the number of partitions of $[n]$ with exactly $k$ blocks
according to the number of symmetric peaks is given by
$$\SP_k(x,q)=x^k(xq+1-x)^{k-1}\prod_{j=1}^{k}W_j(x,q),$$
where
 $$W_j(x,q)=\frac{x(q-1)+(1-x(q-1))W_{j-1}(x,q)}{1-x(1-q)(1-(j-1)x)-xW_{j-1}(x,q)(x(j-1)+q(1-x(j-1)))},$$
 with initial condition $W_0(x,q)=1$.

\end{theorem}
\begin{proof}
In this context we want to use the canonical sequential form of a partition $\pi$ of $[n]$ with exactly $k$ blocks, where $k\geq 2$.
Any partition $\pi$ of $[n]$ with exactly $k$ blocks, can be decomposed as,
 $\pi=\pi'k\omega$, where $\pi'$ is a partition of $[n_1]$ with exactly $k-1$ blocks and $\omega$ is a word of length $n_2$
 over alphabet $[k]$, satisfying $n_1+1+n_2=n$. There are three possibilities: $\omega$ is empty, $\omega$ starts with a letter equal to the last
 letter in $\pi'$, and $\omega$ starts with a letter different from the last letter of $\pi'$.
  Let $W_k(x,q)$ be the generating function for the number of words $\omega$ of length
  $n$ over alphabet $[k]$, according to the statistic of symmetric peaks. By Asakly's result \cite{A}, we have  $$W_k(x,q)=\frac{x(q-1)+(1-x(q-1))W_{k-1}(x,q)}{1-x(1-q)(1-(k-1)x)-xW_{k-1}(x,q)(x(k-1)+q(1-x(k-1)))}.$$
  The corresponding generating functions for the aforementioned cases are:
   $$x\SP_{k-1}(x,q),$$
   $$x^2q\SP_{k-1}(x,q)W_k(x,q),$$
   $$x\SP_{k-1}(x,q)(W_k(x,q)-xW_k(x,q)-1).$$
   This leads to
  $$\SP_k(x,q)=x\SP_{k-1}(x,q)W_k(x,q)(xq+1-x),$$
  with initial conditions,
  $\SP_0(x,q)=1$ and $\SP_1(x,q)=\frac{x}{1-x}$.
  By induction, we obtain the required result.
\end{proof}

Note that, by substituting $q=1$ in Theorem \ref{Th1}, we obtain $\SP_k(x,1)=x^k\prod_{j=1}^{k}W_j(x,1)=x^k\prod_{j=1}^{k}\frac{1}{1-jx}$ which is
the generating function for the number of partitions of $[n]$ with $k$ blocks.
We aim to find the total number of symmetric peaks
over all partitions of $[n]$ with exactly $k$ blocks. We proceed as follows:

\begin{itemize}
 \item First, differentiate the generating function $\SP_k(x,q)$ with respect
 to the variable $q$.
\item Second, substitute $q=1$ in $\frac{\partial }{\partial q}\SP_k(x,q)$
\item Finally, find the coefficients of $x^n$ in $\frac{\partial}{\partial q}\SP_k(x,q)\mid_{q=1}$
to obtain the total number of symmetric peaks over all $P_{n,k}$.
\end{itemize}
\begin{lemma}\label{lem1}
The partial derivative $\frac{\partial }{\partial q}\SP_k(x,q)\mid_{q=1}$ is given by
 \begin{equation}\label{eq1}
 \frac{\partial }{\partial q}\SP_k(x,q)\mid_{q=1}=(k-1)x\SP_k(x,1)+\SP_k(x,1)\sum_{m= 1}^{k}\frac{x^3\binom{m}{2}}{(1-mx)}.
 \end{equation}
\end{lemma}
\begin{proof}
We need to differentiate the generating function $\SP_k(x,q)$ with respect
 to the variable $q$ and then evaluate the result at $q=1$:
 \begin{align*}
  \frac{\partial }{\partial q}\SP_k(x,q)\mid_{q=1}
  &=x^{k+1}(k-1)(xq+1-x)^{k-2}\prod_{j=1}^{k}W_j(x,q)\mid_{q=1}\\
  &+x^k(xq+1-x)^{k-1}\left( \sum_{m=1}^{k}\frac{\partial}{\partial q}W_m(x,q)\prod_{j=1, j\neq m}^{k}W_j(x,q)\right) \mid_{q=1}\\
  &=x^{k+1}(k-1)\prod_{j=1}^{k}W_j(x,1)+x^k\left( \sum_{m=1}^{k}\frac{\frac{\partial}{\partial q}W_m(x,q)}{W_m(x,q)}\prod_{j=1}^{k}W_j(x,q)\right) \mid_{q=1},
 \end{align*}
 this leads to,
 \begin{equation}\label{eq2}
  \frac{\partial }{\partial q}\SP_k(x,q)\mid_{q=1}=x^{k+1}(k-1)\prod_{j=1}^{k}W_j(x,1)+x^k\left( \sum_{m=1}^{k}\frac{\frac{\partial}{\partial q}W_m(x,1)}{W_m(x,1)}\prod_{j=1}^{k}W_j(x,1)\right),
 \end{equation}
 using the facts that,
 \begin{align*}
 \SP_k(x,1)=x^k\prod_{j=1}^{k}W_j(x,1)=x^k\prod_{j=1}^{k}\frac{1}{1-jx},
 \end{align*}
 and
\begin{align*}
\frac{\partial}{\partial q}W_k(x,q)\mid_{q=1}=\frac{x^3\binom{k}{2}}{(1-kx)^2},
\end{align*}
(see Asakly \cite{A} for more details),
together in (\ref{eq2}), we get the required.
\end{proof}
\begin{corollary}\label{cor1}
The total number of symmetric peaks in all partitions of $[n]$ with exactly $k$ blocks is
\begin{equation*}
 (k-1)S_{n-1,k}+\sum_{j=2}^{k}\binom{j}{2}\sum_{i=3}^{n-k}j^{i-3}S_{n-i,k}.
\end{equation*}
\end{corollary}
\begin{proof}
In order to enumerate the total number of all partitions of $[n]$ with exactly $k$ blocks according to
the symmetric peaks, we need to find the coefficients of $x^n$ in (\ref{eq1}).
Due to the fact that,
\begin{align*}
\SP_k(x,1)=\frac{x^k}{\prod_{j=1}^{k}(1-jx)}=\sum_{r\geq k}S_{r,k}x^r.
\end{align*}
we get the required result.
\end{proof}
\subsection{Combinatorial proof}
In this subsection, we want to present a combinatorial proof for Corollary \ref{cor1}. For that, we need the following
definitions: Consider any set partition $\pi = \pi_1 \pi_2 \cdots \pi_n$, represented by its canonical sequence.
We say that $\pi$ contains a {\em rise (descent)} at $i$ if $\pi_i<\pi_{i+1}(\pi_i>\pi_{i+1})$. For instance, for $\pi=1121324323 \in P_{10,4}$, we have $4$ rises at $i=2, 4, 6, 9$ and $4$ descents at $i=3, 5, 7, 8$. Furthermore, we say that $\pi_i$ is a record if $\pi_i>\pi_j$ for  all $j=1, 2, \ldots, i-1$, and $i$ is called the index of the record $\pi_i$, see \cite{ATS}.
\begin{proof}
Let us divide the proof into two parts. In the first part, our focus is on symmetric peaks $\pi_\ell \pi_{\ell+1}\pi_{\ell+2}$ where $\pi_{\ell+1}$ is a record. In the second part, our attention will turn to a symmetric peak, where $\pi_{\ell+1}$ is not a record.

Let $\pi = \pi_1 \pi_2 \cdots \pi_{n-1}$ be a canonical sequential form for $\Pi \in P_{n-1,k}$. Choose any record $\pi_{\ell+1}$ in $\Pi \in P_{n-1,k}$, where $1\leq \ell\leq n-1$, and let $a=\pi_\ell$. It is obvious that we have a rise at $\ell$. Add $\ell+2$ to block $a$ and increase by one any member in their block that is greater or equal to $\ell+2$. As a result, we obtain the canonical sequential form $\pi'=\pi_1\pi_2\cdots\pi_\ell\pi_{\ell+1}a\pi_{\ell+2}\cdots\pi_{n-1}$, where $\pi_{\ell+1}a\pi_{\ell+2}$ forms a symmetric peak. With $k-1$ options for choosing a record, we have a total of $(k-1)S_{n-1,k}$ symmetric peaks.

Considering $i$ and $j$, where $3 \le i \le n-k$ and $2 \le j \le k$, the total number of symmetric peaks at $\ell$ where $\pi_{\ell+1}$ is not a record can be determined by summing the optional partitions in $P_{n,k}$ that can be decomposed as $\pi=\pi''j\alpha\beta$ for each $i$ and $j$. Here, $\pi''$ is a partition with $j-1$ blocks, $\alpha$ is a word of length $i$ over alphabet $[j]$ whose last three letters form a symmetric peak, and $\beta$ may be an empty word. There are ${j \choose 2}j^{i-3}S_{n-i,k}$ members of $P_{n,k}$ with this form. This is because there are $j^{i-3}$ choices for the first $i-3$ letters of $\alpha$, ${j \choose 2}$ for the final three letters in $\alpha$ (representing a symmetric peak), and $S_{n-i,k}$ choices for the remaining letters $\pi=\pi''j\beta$, ensuring that they form a partition of an $(n-i)$-set into $k$ blocks.
\end{proof}

\section{Counting non-symmetric peaks}
\subsection{The ordinary generating function for the number of partitions with exactly $k$ blocks according to the number of
non-symmetric peaks}
Let $\nonsympeak(\pi)$ denote the number of non-symmetric peaks in partition $\pi$.
Let $\NSP_k(x,q)$ be the ordinary generating function for the number of partitions of $[n]$ with exactly $k$ blocks
according to the number of non-symmetric peaks, that is
$$\NSP_k(x,q)=\sum_{n\geq k}x^n\sum_{\pi \in P_{n,k}}q^{\nonsympeak(\pi)}.$$

\begin{theorem}\label{Th2}
 The generating function for the number of partitions of $[n]$ with exactly $k$ blocks
according to the number of non-symmetric peaks is given by
$$\NSP_k(x,q)=x^k\prod_{i=1}^{k}\widetilde{W}_i(x,q)\prod_{j=3}^{k}\left((j-2)xq+1
-(j-2)x\right)$$
where
 $$\widetilde{W}_j(x,q)=\frac{x(q-1)+(1-x(q-1))\widetilde{W}_{j-1}(x,q)}{1-x(1-q)(1-2x)-x\widetilde{W}_{j-1}(x,q)(2x+q(1-2x))},$$
 with initial condition $\widetilde{W}_0(x,q)=1$.
\end{theorem}
\begin{proof}
Let $\pi$ be any partition of $[n]$ with exactly $k$ blocks, where $k\geq 3$. It can be decomposed as
 $\pi=\pi'k\omega$, where $\pi'$ is a partition of $[n_1]$ with exactly $k-1$ blocks, and $\omega$ is a word of length $n_2$
 over alphabet $[k]$ satisfying $n_1+1+n_2=n$. There are three possibilities: $\omega$ is empty, $\omega$
 starts with a letter different from the last
 letter in $\pi'$, and $\omega$ starts with a letter equal to the last letter of $\pi'$.
  Let $$\widetilde{W}_k(x,q)=\frac{x(q-1)+(1-x(q-1))\widetilde{W}_{k-1}(x,q)}{1-x(1-q)(1-2x)-x\widetilde{W}_{k-1}(x,q)(2x+q(1-2x))},$$
   be the generating function for the number of words $\omega$ of length
  $n$ over alphabet $[k]$ according to the statistic non-symmetric peaks (see \cite{A}).
  Then the corresponding generating functions for the above cases respectively are,
   $$x\NSP_{k-1}(x,q),$$
   $$x^2(j-2)q\NSP_{k-1}(x,q)\widetilde{W}_k(x,q),$$
   $$x\NSP_{k-1}(x,q)(\widetilde{W}_k(x,q)-x(j-2)\widetilde{W}_k(x,q)-1).$$
   This leads to
  $$\NSP_k(x,q)=x\NSP_{k-1}(x,q)\widetilde{W}_k(x,q)(xq(j-2)+1-x(j-2)),$$
  with initial conditions,
  $\NSP_0(x,q)=1$, $\NSP_1(x,q)=\frac{x}{1-x}$ and $\NSP_2(x,q)=\frac{x^2}{(1-x)(1-2x)}$.
  By induction we obtain the required result.
\end{proof}

By substituting $q=1$ in Theorem \ref{Th2}, we obtain $\NSP_k(x,1)=x^k\prod_{j=1}^{k}W_j(x,1)=x^k\prod_{j=1}^{k}\frac{1}{1-jx}$ which is
the generating function for the number of partitions of $[n]$ with $k$ blocks.

 Our goal is to find the total number of the non-symmetric peaks over all partitions of $[n]$ with exactly $k$ blocks. To achieve this, we repeat the same steps as presented in Lemma {\ref{lem1}}.
\begin{lemma}\label{lem2}
The partial derivative $\frac{\partial }{\partial q}\NSP_k(x,q)\mid_{q=1}$ is given by
 \begin{align*}
 \frac{\partial }{\partial q}\NSP_k(x,q)\mid_{q=1}=\binom{k-1}{2}x\NSP_k(x,1)+\NSP_k(x,1)\sum_{m=3}^{k}\frac{2x^3\binom{m}{3}}{(1-mx)}.
 \end{align*}
\end{lemma}
\begin{proof}
By differentiating the generating function $\NSP_k(x,q)$ with respect
 to the variable $q$ and then evaluating the result at $q=1$ we obtain,
 \begin{align*}
  \frac{\partial }{\partial q}\NSP_k(x,q)\mid_{q=1}
  &=\left(x^k\sum_{m=1}^{k}\frac{\partial}{\partial q}\widetilde{W}_m(x,q)\prod_{i=1, i\neq m}^{k}\widetilde{W}_i(x,q)\prod_{j=3}^{k}((j-2)xq+1-(j-2)x)\right)\mid_{q=1}\\
  &+\left(x^k\prod_{i=1}^{k}{\widetilde{W}_i(x,q)}\sum_{m=3}^{k}(m-2)x\prod_{j=3, j\neq m}^{k}((j-2)xq+1-(j-2)x)\right)\mid_{q=1}\\
  &=x^k\sum_{m=1}^{k}\frac{\partial}{\partial q}\widetilde{W}_m(x,q)\mid_{q=1}\prod_{i=1, i\neq m}^{k}\widetilde{W}_i(x,1)+x^k\prod_{i=1}^{k}{\widetilde{W}_i(x,1)}\sum_{m=3}^{k}(m-2)x\\
  &=x^k\sum_{m=1}^{k}\frac{\frac{\partial}{\partial q}\widetilde{W}_m(x,q)\mid_{q=1}}{\widetilde{W}_m(x,1)}\prod_{i=1}^{k}\widetilde{W}_i(x,1)+x\frac{(k-1)(k-2)}{2}x^k\prod_{i=1}^{k}{\widetilde{W}_i(x,1)}.
 \end{align*}
 According to Asakly \cite{A}, $\frac{\partial}{\partial q}\widetilde{W}_k(x,q)\mid_{q=1}=\frac{2x^3\binom{k}{3}}{(1-kx)^2}$ and the equalities $\widetilde{W}_k(x,q)\mid_{q=1}=\frac{1}{1-kx}$ and $\NSP_k(x,1)=x^k\prod_{i=1}^{k}{\widetilde{W}_i(x,1)}$ lead to the required result.
\end{proof}
\begin{corollary}\label{cor2}
The total number of non-symmetric peaks in all partitions of $[n]$ with exactly $k$ blocks is
\begin{equation*}
 \binom{k-1}{2}S_{n-1,k}+2\sum_{j=3}^{k}\binom{j}{3}\sum_{i=3}^{n-k}j^{i-3}S_{n-i,k}.
\end{equation*}
\end{corollary}
\begin{proof}
By finding the coefficients of $x$ in $\frac{\partial }{\partial q}\NSP_k(x,q)\mid_{q=1}$ we get the result.
\end{proof}
\subsection{Combinatorial proof}
In this subsection, we want to present a combinatorial proof for Corollary (\ref{cor2}).
\begin{proof}
Let us focus on non-symmetric peaks $\pi_\ell \pi_{\ell+1}\pi_{\ell+2}$ where $\pi_{\ell+1}$ is a record. Consider $\pi = \pi_1 \pi_2 \cdots \pi_{n-1}$ as a canonical sequential form for $\Pi \in P_{n-1,k}$. Choose any record $\pi_{\ell+1}$ in $\Pi \in P_{n-1,k}$, where $1\leq \ell\leq n-1$. It is obvious that we have a rise at $\ell$. Add $\ell+2$ to any block $a$, where $a<\pi_{\ell+1}$ and $a\neq \pi_\ell$, by increasing any member in their block that is greater or equal to $\ell+2$ by one. As a result, we obtain the canonical sequential form $\pi'=\pi_1\pi_2\cdots\pi_\ell\pi_{\ell+1}a\pi_{\ell+2}\cdots\pi_{n-1}$, where $\pi_{\ell+1}a\pi_{\ell+2}$ forms a non-symmetric peak. For each chosen record $\pi_{\ell+1}$, there are $\pi_{\ell+1}-2$ possible choices for the block $a$. The sum of these choices, considering all optional records from $3$ to $k$ (excluding $1$ and $2$ since they do not form non-symmetric peaks), gives $\frac{\left(k-1\right)\left(k-2\right)}{2}$ possibilities. As a result, this leads to $\binom{k-1}{2}S_{n-1,k}$ non-symmetric peaks.

Let us focus on non-symmetric peaks $\pi_\ell \pi_{\ell+1}\pi_{\ell+2}$, where $\pi_{\ell+1}$ is not  a record. Considering $i$ and $j$, where $3 \le i \le n-k$ and $2 \le j \le k$, the total number of non-symmetric peaks at $\ell$ where $\pi_{\ell+1}$ is not a record can be determined by summing the optional partitions in $P_{n,k}$ that can be decomposed as $\pi=\pi''j\alpha\beta$ for each $i$ and $j$. Here, $\pi''$ is a partition with $j-1$ blocks, $\alpha$ is a word of length $i$ over alphabet $[j]$ whose last three letters forming a non-symmetric peak, and $\beta$ may be an empty word. There are $2{j \choose 3}j^{i-3}S_{n-i,k}$ members of $P_{n,k}$ with this form. This is because there are $j^{i-3}$ choices for the first $i-3$ letters of $\alpha$, $2{j \choose 3}$ for the final three letters in $\alpha$ (representing a non-symmetric peak), and $S_{n-i,k}$ choices for the remaining letters $\pi=\pi''j\beta$, ensuring that they form a partition of an $(n-i)$-set into $k$ blocks.
\end{proof}

\end{document}